\documentclass[article,oneside,reqno,11pt,pdftex]{memoir}
 \usepackage{color}
 \definecolor{darkblue}{rgb}{0,0.08,0.50} 
\usepackage{amsmath,amsthm,amssymb}
 \usepackage[british]{babel} 
 \usepackage[bookmarksnumbered=true,colorlinks=true, linkcolor=darkblue,
              citecolor=darkblue, urlcolor=darkblue, pdfpagemode=none]
             {hyperref} 
 \usepackage{memhfixc} 

\DeclareFontFamily{OT1}{rsfs}{}
\DeclareFontShape{OT1}{rsfs}{n}{it}{<-> rsfs10}{}
\DeclareMathAlphabet{\curly}{OT1}{rsfs}{n}{it}

\newcommand{\trace}{\operatorname{tr}}
\newcommand{\Scal}{\operatorname{Scal}}
\renewcommand{\tt}{\ttfamily}
\newcommand{\orb}{\mathrm{orb}} 
\newcommand{\Ric}{\operatorname{Ric}}
\newcommand{\average}{\mathrm{av}}

\newcommand\C{\mathbb C}

\newcommand\R{\mathbb R}

\newcommand\ZZ{\mathbb Z}
\newcommand\PP{\mathbb P}

\newcommand{\K}{\mathcal K}
\newcommand{\B}{\mathcal B}

\renewcommand\O{\mathcal O}

\newcommand{\del}{\partial}

\newcommand{\ord}{\operatorname{ord}}

\makeatletter \@addtoreset{equation}{chapter} \makeatother

\newtheorem{thm}[equation]{Theorem}
\newtheorem{lem}[equation]{Lemma}

\newtheorem{cor}[equation]{Corollary}
\newtheorem{prop}[equation]{Proposition}

\renewcommand{\Re}{\operatorname{Re}}
\theoremstyle{definition}

\newtheorem{defn}[equation]{Definition}

\newtheorem{rmk}[equation]{Remark}
\newtheorem*{rmk*}{Remark}

\begin{document}

\setcounter{secnumdepth}{1}
\pagestyle{myheadings}
\title{Weighted Bergman kernels on orbifolds}
\author{Julius Ross and Richard Thomas}
\date{}


\hypersetup{
pdfauthor = {J. Ross and R. P. Thomas},
pdftitle = {Weighted Bergman kernels on orbifolds},
pdfkeywords = {MSC  32A25},
pdfcreator = {LaTeX with hyperref package}
}


\bibliographystyle{amsalpha_ross}
\maketitle
\begin{abstract}
  We describe a notion of ampleness for line bundles on orbifolds with cyclic quotient singularities that is related to embeddings in weighted projective space, and  prove a global asymptotic expansion for a weighted Bergman kernel associated to such a line bundle.
\end{abstract}


\newcommand{\sumskipp}[3]{
\sum_{#1\equiv #2} #3
}

\newcommand{\sumskip}[1]{\sumskipp{i}{u}{#1}}

\chapter{Introduction}

Let $(X,\omega)$ be a compact $n$-dimensional K\"ahler manifold and $L$ be a positive line bundle on $X$ equipped with a hermitian metric $h$ whose curvature form is $-2\pi i\omega$.    These induce an $L^2$-metric on the space of sections $H^0(L^k)$, and given an orthonormal basis $\{s_\alpha\}$ the Bergman kernel is  the smooth function
$$
 B_k(x) = \sum_{\alpha} |s_\alpha(x)|^2,\label{eq:defbergmanL}
$$
where $|s_{\alpha}(x)|$ is the pointwise norm induced by $h$ (more precisely, its
$k$th power considered as a metric on $L^k$). More invariantly, the
$L^2$-metric on $H^0(L^k)$ induces a Fubini-Study metric on $\PP(H^0(L^k)^*)\supset X$
and $\O(1)$ over it; restricting to $X$ we get a metric $h_{FS}$ on $\O_X(1)\cong L^k$
which may not equal $h$. Then $B_k=h/h_{FS}$ is their ratio.

Work of Fefferman \cite{fefferman:74:bergm_kernel_bihol_mappin_pseud_domain},
Yau \cite{yau:86:nonlin_analy_in_geomet},
Tian \cite{tian:90:set_polar_kaehl_metric_algeb_manif}, 
Zelditch \cite{zelditch:98:szeg_o_kernel_theor_tian}, Catlin
\cite{catlin:99:bergm_kernel_theor_tian} and Ruan \cite{ruan:98:canon_coord_bergm_bergm_metric}
describes the asymptotic behaviour of this
function for large $k$:  there exist smooth functions $b_1,\ldots,b_N$ and a global expansion
\begin{equation}
 B_k  = k^n + b_1k^{n-1} + \cdots +b_N k^{n-N} + O(k^{n-N-1}) \label{eq:manifoldexpansion}
\end{equation}
for $k\gg 0$, where $O(k^{n-N-1})$ can even be taken in the $C^\infty$ norm.
Moreover, the $b_j$ can be expressed in terms of the derivatives of the
metric \cite{lu:00:lower_order_terms_asymp_expan}. In particular, $b_1$ is half the scalar curvature of $\omega$, leading to the importance of this expansion in
the theory of constant scalar curvature K\"ahler metrics \cite{yau:86:nonlin_analy_in_geomet,
donaldson(01):scalar_curvat_projec_embed}. \medskip

Throughout this paper, $X$ will instead be an orbifold (usually compact, except when we
work locally) with cyclic stabiliser groups and an orbifold line bundle $L$ over it.
We let $\ord_x$ denote the order of a point $x\in X$: the size of the cyclic stabiliser
group of any lift of this point in an orbifold chart. When $X$ is compact, $\ord(X)$
denotes the least common multiple of this finite collection of integers.

Then the same expansion of $B_k$ holds away from the
orbifold locus \cite{dai-liu-ma(06):asymp_expansion_bergman_kernel,
song:szegoe_kernel_orbif_circl_bundl}. And while the precise form of the expansion over the orbifold locus is given in \cite{dai-liu-ma(06):asymp_expansion_bergman_kernel}, the result is not smooth; see Remark \ref{rmk:DLM}.
 
We illustrate the situation locally using the simplest example: the
orbifold $\C/(\ZZ/2)$ with local coordinate $z$
on $\C$ acted on by $\ZZ/2$ via $z\mapsto-z$. Then $x=z^2$ is a local
coordinate on the quotient thought of as a manifold. An ordinary (non-orbifold) line
bundle is one pulled back from
the quotient, i.e., one which
has trivial $\ZZ/2$-action upstairs when considered as a trivial line
bundle there. This has invariant sections $\C[x]=\C[z^2]$. We do not consider such line
bundles, as their sections only see the quotient manifold $\mathrm{Spec}\,\C[x]$ (to which
the usual Bergman expansion applies), missing
the extra functions of $\sqrt x=z$ that the orbifold sees. 

Instead we use the nontrivial \emph{orbifold line bundle} given by the 
trivial line bundle upstairs with nontrivial $\ZZ/2$-action
(acting as $-1$ on the trivialisation). This has invariant sections $\sqrt x\,\C[x]=z\C[z^2]$,
while its square has trivial $\ZZ/2$-action and has sections $\C[x]=\C[z^2]$
as above. Therefore, the sections of its powers generate the entire ring of functions
$\C[\sqrt x]=\C[z]$ upstairs, and see the full orbifold structure.

\begin{defn} \cite{ross_thomas:stabil_orbif} \label{defn:locallyample}
An orbifold line bundle $L$ over a cyclic orbifold $X$ is \emph{locally ample}
if in an orbifold chart around $x\in X$, the stabiliser group acts
faithfully on the line $L_x$.  We say $L$ is \emph{orbi-ample} if it is both locally
ample and globally positive. (That is, $L^{\ord(X)}$ is ample in the usual sense when
thought of as a line bundle on the underlying space of $X$; by the Kodaira-Baily embedding
theorem \cite{baily:57:imbed_v_manif_in_projec_space} one can equivalently ask that $L$
admits a hermitian metric with positive curvature.)
\end{defn}

In \cite[sections 2.5--2.6]{ross_thomas:stabil_orbif} this notion of ampleness is shown to be equivalent to the existence of an orbifold embedding
$$ (X,L) \hookrightarrow (\mathbb{WP},\O(1))$$
of $X$ into a weighted projective space $\mathbb W\mathbb P$. That is, the orbifold structure
of $X$ is pulled back from that of $\mathbb{WP}$, and $L$ is the pullback of the orbifold
hyperplane bundle $\mathcal O_{\mathbb{WP}}(1)$.

Just as in our example above, for any orbifold line bundle $L$ which has nontrivial stabiliser
action on the line $L_x$ over an orbifold point $x$, the sections of $L$
(i.e., the invariant
sections upstairs) all vanish at $x$. Therefore, for most powers $L^k$ of an orbi-ample
line bundle, the Bergman kernel $B_k$ will vanish at orbifold points.
Conversely, for
powers divisible by $\ord_x$, $B_k(x)$ will be nonzero, so in general $B_k(x)$ will have
some periodic behaviour in $k$ at orbifold points. The purpose
of this paper is to get a smooth global expansion by taking weighted
sums of Bergman kernels associated to various powers of $L$, flattening out the periodicity.
In the companion paper \cite{ross_thomas:stabil_orbif} we apply this to orbifold K\"ahler
metrics of constant scalar curvature and their relationship to stability. \smallskip

The Bergman kernel measures the local density of holomorphic
sections, so to get an expansion on an orbifold, we first
discuss the general local situation---the quotient of an open set in $\C^n$ by a linear action
of the cyclic group $\mathbb Z/m:=\mathbb Z/m\mathbb Z$. The reader will not lose much by considering
only the $\C/(\mathbb Z/2)$ example above, with the orbifold line bundle with
trivialisation which has weight $\pm1$ under the $\mathbb Z/2$-action.

By the local ampleness condition, there exists an identification between $\mathbb Z/m$ and
the group of $m$th roots of unity with respect to which there exists a local equivariant
(not invariant!) trivialisation of $L$ of weight $-1$. Then, with respect to this
trivialisation, sections of $L^k$ downstairs, i.e., invariant sections upstairs on $\C^n$,
are the same thing as functions upstairs of $\ZZ/m$-weight $k$ mod $m$. (In
particular,
they vanish at the origin unless $k\equiv 0$ mod $m$.)

We consider the weighted Bergman kernel
\begin{equation}\label{eq:defBorb}
B^{\orb}_k:=\sum_{i} c_i B_{k+i}\,, 
\end{equation}
where $c_i$ are positive constants and $i$ runs over some fixed finite index set of nonnegative
integers.  This is the global expression downstairs on the quotient; upstairs locally
we are taking a similar expression, but only include the sections of $L^{k+i}$ that have
$\ZZ/m$-weight $0$ mod $m$ (or, using the trivialisation, the functions of weight
$k+i$ mod $m$).

Since functions of nonzero weight vanish at the orbifold point, the ordinary
Bergman kernel at the origin upstairs is equal to the sum of the $|s_\alpha|^2$ over
an orthonormal basis of \emph{only the invariant sections}; i.e., it equals
the ``downstairs'' Bergman kernel $B_{k+i}$ at the origin. (Using the
trivialisation, this corresponds to summing over only the functions of weight $k+i$.)
Therefore, by
\eqref{eq:manifoldexpansion} applied to the ordinary upstairs Bergman kernel at the
origin, the downstairs kernel $B_{k+i}(0)$ has the asymptotic expansion for fixed $i$ and as $k$ tends to infinity
$$
(k+i)^n +\frac{1}{2}\Scal(\omega)(k+i)^{n-1} + \cdots.
$$
Summing over $i$, we find that in \eqref{eq:defBorb} the functions upstairs of $\ZZ/m$-weight
$u$ contribute, to leading order in $k$, 
\begin{equation} \label{contrib}
\sumskipp{i}{u-k}{c_i k^n}
\end{equation}
to $B_k^{\orb}(0)$. (Here and throughout the paper, the subscript means that the sum is
taken over all $i$ equal to $u-k$ mod $m$.) 

For instance, in our $\C/(\ZZ/2)$ example above with the nontrivial orbifold line
bundle, the functions of $\ZZ/2$-weight $0$ (respectively $1$) mod $2$ contribute
only to the terms $B_{k+i}(0)$ with $k+i$ even (respectively odd).
To even things out and ensure that $B_k^{\orb}$ has a smooth global
expansion, we want
to use \emph{all} local functions of all $\ZZ/m$-weights equally (thus returning us
to something close to the original Bergman kernel \emph{upstairs}). So we
choose the $c_i$ so that the functions of each $\ZZ/m$-weight contribute to the sum
\eqref{eq:defBorb} with the same coefficient.

We therefore want for each $k$ for \eqref{contrib} to be equal for all $u$,
i.e.,
$$\sumskipp{i}{u}{c_i} \quad\text{is independent of }u.$$
More generally, if $N\le n$ and all of the coefficients of $k^n,\ldots,k^{n-N}$ in
\begin{equation}\label{eq:starstar}
\sum_{i} c_i (k+i)^n      
\end{equation}
have equal contributions in each weight (mod $m$), i.e.,
\begin{equation}\label{eq:starstarstar}
\sumskipp{i}{u}{i^p c_i} \quad \text{is independent of } u \text{ for }
p=0,\ldots,N,
\end{equation}
then the same will be true of each
$\sum_{i} c_i (k+i)^{n-j}$ for $j=0,\ldots, N$,  and  $B_k^{\orb}$ will
admit an asymptotic expansion in $k$ of order $N$. In fact, we impose a slightly stronger condition than
\eqref{eq:starstarstar} in order to get the expansion at the $C^r$ level.


\begin{thm}\label{thm:mainexpansion}
  Let $(X,\omega)$ be a compact $n$-dimensional K\"ahler orbifold with
  cyclic quotient singularities, and $L$ be an orbi-ample line bundle on $X$ equipped with a hermitian metric whose curvature form is $-2\pi i\omega$.     Fix $N,r\ge 0$ and suppose $c_i$ are a finite number of positive constants chosen so that if $X$ has an orbifold point of order $m$, then
\begin{eqnarray}\label{eq:conditiononalpha}
\frac{1}{m}\sum_{i} i^p c_i &=& \mathop{\sum {i^pc_i}}_{i \equiv u \text{ mod }m}\quad \text{for all } u\text{ and  } p=0,\ldots,N+r.
\end{eqnarray}
 Then  the function 
$$ B^{\orb}_k:=\sum_{i} c_i B_{k+i}$$
admits a global $C^{r}$-expansion of order $N$. That is, there exist smooth functions $b_0,\ldots,b_{N}$ on $X$ such that
\[ B^{\orb}_k  =  b_0 k^n + b_ 1k^{n-1} + \cdots + b_N k^{n-N} + O(k^{n-N-1}),\]
where the $O(k^{n-N-1})$ term is to be taken in the $C^{r}$-norm.
Furthermore, the  $b_j$ are universal polynomials in the constants $c_i$
and the derivatives of the K\"ahler metric $\omega$; in particular,
$$
  b_0 = \sum_{i} c_i\quad \text{ and }\quad b_1 = \sum_{i} c_i \left(ni + \frac{1}{2}\Scal(\omega)\right),
$$
where $\Scal(\omega)$ is the scalar curvature of $\omega$.
\end{thm}


\begin{rmk}
If condition \eqref{eq:conditiononalpha} holds for some $m$, then it also holds if $m$ is replaced by any factor;  hence for the theorem it is sufficient to assume it holds when $m=\ord(X)$ is the lowest common multiple of the orders of the stabiliser groups
of all points of $X$.  So in particular, the  $N=r=0$ case yields a top-order $C^0$-expansion for $B_k^{\orb}$ when $c_i=1$ for $i=1,\ldots,\ord(X)$ and $c_i=0$ otherwise.   It turns out (Lemma \ref{lem:lemmaonconstants}) that \eqref{eq:conditiononalpha} is equivalent to asking that the function $\sum_i c_i z^i$ has a root of order $N+r+1$ at each nontrivial $m$-th root of unity. Thus there is little loss in assuming the $c_i$ are defined by
\[\sum_{i}c_i z^i :=  (z^{m-1} + z^{m-2} + \cdots + 1)^{N+r+1}.\]
In fact, the condition \eqref{eq:conditiononalpha} is natural in the sense that
it is also necessary for the existence of an expansion; see Remark
\ref{rmk:necessityofconstants}.    Also, if $X$ is a manifold, then there is no condition on the $c_i$, and the theorem is equivalent to the expansion \eqref{eq:manifoldexpansion}
for manifolds stated above.
\end{rmk}


\begin{rmk}
By a $C^{r}$-expansion of order $N$ we mean there is a constant $C_{r,N}$ such that
$$ \left\|  B_k^{\orb} - \left(b_0 k^n + b_1 k^{n-1} + \dots + b_N k^{n-N}\right)\right\|_{C^{r}} \le C_{r,N} k^{n-N-1},$$
 where the norm $\|f\|_{C^{r}}$ is the sum of the supremum norms of $f$ and
 its first $r$ derivatives over all of $X$ in the orbifold sense (i.e., we measure the norm of derivatives using the metric upstairs on orbifold charts).    The theorem generalises to hermitian metrics $h$ on $L$ with curvature form $-2\pi i\omega_h$ not necessarily equal to $\omega$, the only change being that the coefficients in the expansion become 
\cite[Theorem 4.1.3, Remark 5.4.13]{ma_marinescu:07:holom_morse_inequal_bergm_kernel}
 \begin{eqnarray}
   b_0 &=&\frac{\omega_h^n}{\omega^n} \sum_{i} c_i,\nonumber \\
 b_1 &=& \frac{\omega_h^n}{\omega^n}\sum_{i}c_i[ni + \trace_{\omega_h}(\Ric(\omega))-\frac{1}{2}\Scal(\omega_h)].
\end{eqnarray}
Moreover, the constants $C_{r,N}$ can be taken to be uniform as $(h,\omega)$ runs over a compact set.
\end{rmk}


The strategy of our proof is to use known results on manifolds.  For global quotients one can obtain what we want rather easily by averaging the expansion upstairs. To apply
this to general orbifolds, we need to be able to work locally, which is what the approach
of Berman-Berndtsson-Sj\"ostrand \cite{berman_berndtsson_sjoestrand:08:direc_approac_to_bergm_kernel}
to the expansion \eqref{eq:manifoldexpansion} allows us to do. This starts by proving a local expansion for the Bergman kernel, and then uses the H\"ormander estimate to pass to a global expansion.  In Section \ref{sec:local} we average these local expansions to get a local orbifold expansion, and use it in Section \ref{sec:global} to get a  global expansion in much the same way as for manifolds. 
\begin{rmk}\label{rmk:DLM}
There are several approaches to the analysis that one may take to prove the above result. For example, once one arranges the Bergman kernels in the 
combination of Theorem 1.7 the main result follows quickly
from the precise control of the singularity of the Bergman kernel expansion over the orbifold locus proved in
\cite[(5.25)]{dai-liu-ma(06):asymp_expansion_bergman_kernel}; see
\cite{xianzhe:11:remar_weigh_bergm_kernel_orbif} for details.
\end{rmk}
\smallskip

One consequence is a Riemann-Roch expression for orbi-ample line bundles obtained
by integrating the expansion over  $X$.
The difference between this and the general Kawasaki-Riemann-Roch formula for orbifolds
\cite{kawasaki:79:rieman_roch_theor_for_compl_v_manif} is that the periodic term coming
from the orbifold singularities vanishes due to the weighted sum over $i$.  We denote the orbifold canonical bundle by $K_{\orb}$  (defined, for instance, in \cite[section 2.4]{ross_thomas:stabil_orbif}).
\begin{cor}\label{cor:rroch}
Let $X$ be an $n$-dimensional orbifold with cyclic quotient singularities and $L$ be an orbi-ample line bundle.   Let $m=\ord(X)$ and suppose
  $$
\frac{1}{m}\sum_{i} c_i = \mathop{\sum {c_i}}_{i \equiv u \text{ mod m}} \quad \text{ and }\,\,\quad
\frac{1}{m}\sum_{i} ic_i = \mathop{\sum {ic_i}}_{i \equiv u \text{ mod m}} \quad \text{ for all  } u.
$$
Then 
\[ \sum_i c_i h^0(L^{k+i}) =  a_0 k^n + a_1 k^{n-1} + O(k^{n-2})  \]
with
\begin{eqnarray*}
 a_0 &=& \frac{\sum_i c_i}{n!} \int_X c_1(L)^n,    \\
 a_1 &=& \frac{\sum ic_i}{(n-1)!} \int_X c_1(L)^n - \frac{\sum_i c_i}{2(n-1)!} \int_X c_1(K_{\orb})c_1(L)^{n-1}.
\end{eqnarray*}
\end{cor}
\begin{proof}
  This follows from $\int_XB_k\frac{\omega^n}{n!}=h^0(L^k)$ and the fact that the integral of the scalar curvature of any orbifold K\"ahler metric is the topological quantity $-\frac{1}{(n-1)!}\int_X c_1(K_{\orb})c_1(L)^{n-1}$.
\end{proof}

As an indication of the connection with constant scalar curvature metrics,
we give an analogue of Donaldson's Theorem \cite[theorem 2]{donaldson(01):scalar_curvat_projec_embed} for orbifolds (for a more applicable generalisation, see \cite[theorem 5.1] {ross_thomas:stabil_orbif}).

\begin{cor}
Suppose that \eqref{eq:conditiononalpha} is satisfied for $N=1$ and $r=2$. For each $k\gg
0$, let $h_k$ be a hermitian metric on $L$ with curvature $-2\pi i\omega_k$.  Consider the
diagonal sequence of Bergman kernels $B_k^{\orb}$ associated to the metrics $(h_k,\omega_k)$
and the power $L^k$ of $L$, and suppose that for each $k$ this function is
constant over $X$.  If also $\omega_k$ converges in $C^2$ to a K\"ahler metric $\omega_{\infty}$, then $\Scal(\omega_{\infty})$
is constant.
\end{cor}

\begin{proof}\newcommand{\volume}{\operatorname{vol}}
 Let $\volume: =\int_X \frac{c_1(L)^n}{n!}$ be the volume of $L$.  Notice
 that since the Bergman kernel is constant, by the previous corollary and
 integrating over $X$, we have $\volume B_k^{\orb}=  a_0 k^n + a_1 k^{n-1} + O(k^{n-2})$ with $a_i$ as above.   Hence, applying Theorem \ref{thm:mainexpansion} to the pair $(h_k,\omega_k)$,
\[\frac{1}{\volume} (a_0 k^n + a_1 k^{n-1} + O(k^{n-2})) = B_k^{\orb}=b_0k^n + b_1k^{n-1} + O(k^{n-2}),\] 
where the $O(k^{n-2})$ term is independent of $(h_k,\omega_k)$ since they
lie in a compact set.  Furthermore, $b_0 = \sum_i c_i=\frac{a_0}{\volume}$ and $b_1 = \sum_i c_i (ni + \frac{1}{2}\Scal(\omega_k))$.  Thus the $k^n$ terms cancel,  so $b_1$ tends to $a_1\volume^{-1}$ uniformly over $X$ as $k$ tends to infinity.  Since $\omega_k$ tends to $\omega_{\infty}$ in $C^2$, $\Scal(\omega_k)$ tends to $\Scal(\omega_{\infty})$  and thus $\Scal(\omega_{\infty})$ is constant.
\end{proof}


\subsection*{Interpretation using peaked sections}

Intuitively, one can understand what is going on as follows. On a manifold
$B_k(x)$ is the pointwise
norm square of a \emph{peaked section} $s_x$ at $x$ of unit $L^2$-norm. This
is glued from zero outside a small ball of radius $O(k^{-1/2})$
centred on $x$ and the standard local model (with Gaussian norm) inside the
ball. (If $s_x$ is constructed to have zero $L^2$-inner product with any
section vanishing at $x$, then the Bergman kernel at $x$ is \emph{precisely}
$|s_x(x)|^2$; if not, then this is still a good enough approximation for the
discussion here.)

To describe the orbifold case, consider the simple model consisting of a
single orbifold chart $\C\to\C/(\mathbb Z/m)$.  Suppose that $x$ is a smooth point, which in this model case is any point other than the origin.  Then for large $k$ we can pick a ball of radius $O(k^{-1/2})$ centred at $x$ and disjoint from the origin. It therefore admits a copy upstairs in $\C$, i.e., a ball around one of its preimages.   Then just as in the smooth case, there is a peaked section supported in this ball, and moving it around $\C$ by the group action gives an invariant section with
$m$ peaks. This can then be equivariantly glued to zero, giving an invariant section that has $L^2$-norm approximately $m$, but considered as a section downstairs it has norm $1$ (since the volume downstairs is defined to be the volume upstairs divided by the order of the chart). Therefore its pointwise norm squared at $x$ is approximately the value of the Bergman kernel.

Now instead consider an orbifold point, which here is the origin.  A section upstairs on $\C$, peaked about the origin, can be
equivariant \emph{only} when $k\equiv0$ mod $m$. Therefore for most $k$, the
Bergman kernel will have value zero at the origin. But when $k$ is a
multiple of $m$, we get an invariant section with $L^2$-norm 1 upstairs, and
therefore $L^2$-norm only $\frac1m$ when considered as an orbifold section
downstairs. Therefore we multiply the section by $\sqrt m$ to get one of
unit $L^2$-norm, and find that the Bergman kernel of the orbifold at the
origin is approximately $m$ times as big as at manifold points.

Thus when we sum the Bergman kernels over a period $L^k,L^{k+1},\ldots$ of
length at least $m$ we stand a chance of them averaging out to something
which is again approximately constant to top order in $k$. The condition
that they do is that the average of the coefficients $c_i$ with which we
make each $L^{k+i}$ contribute is equal to the sum of the coefficients $c_i$
which contribute at the origin (i.e., those for which $k+i\equiv 0$ mod $m$).
But this is precisely the $p=0$ condition in \eqref{eq:conditiononalpha}.   A similar analysis for $p=1$ means that we can ensure that we smooth out all of the
$i$-dependence of the orbifold Bergman kernel except the part coming from
scalar curvature.\medskip

\noindent\textbf{Conventions:} We refer to \cite{ross_thomas:stabil_orbif} for detailed
definitions
and conventions. For this paper we need only that an orbifold with cyclic stabiliser
groups is an analytic space covered by charts
of the form $U\to U/G\subset X$, where $U$ is an open set in $\mathbb C^n$ and $G$ is
a finite cyclic group acting effectively and \emph{linearly} on $U$; see for example
\cite[section 2.1]{ross_thomas:stabil_orbif}.
It is important for our applications that we allow orbifold structure in codimension
one.  Thus if $m\ge 2$, then $\mathbb C/(\mathbb Z/m)$ is considered a distinct
orbifold from $\mathbb C$ even though they have the same underlying space.
Quantities on $X$ (hermitian metrics, K\"ahler metrics, sections, supremum
norms, etc.) are always taken in the orbifold sense, as an invariant quantity upstairs in orbifold charts. An orbifold line bundle is an equivariant line bundle on orbifold charts; the
gluing condition is described in \cite{ross_thomas:stabil_orbif}. \smallskip

\subsection*{Acknowledgments}
We thank Florin Ambro, Bo Berndtsson, and Yan\-ir Rubenstein for helpful conversations.
We also wish to acknowledge our debt to the paper
\cite{berman_berndtsson_sjoestrand:08:direc_approac_to_bergm_kernel}, which
allowed for a significant improvement of ours.  JR received support from NSF Grant DMS-0700419 and Marie Curie Grant PIRG-GA-2008-230920. RT held a Royal Society University Research Fellowship while this work was carried out.

\chapter{Bergman Kernels}\label{sec:bergman}

From now on, let $(X,\omega)$ be a compact $n$-dimensional K\"ahler
orbifold and $L$ be an orbifold line bundle with a hermitian metric $h$ whose curvature form is $-2\pi i\omega$.  By abuse of notation, we denote the induced metric on $L^k$ also by $h$, which along with the volume form defined by $\omega$ yields an $L^2$-inner product on $H^0(L^k)$ given by

\[ (s,t)_{L^2} = \int_X  (s,t)_{h} \frac{\omega^n}{n!} \quad \text{for }s,t\in H^0(L^k).\]
The Bergman kernel of $L^k$ is  defined to be
$B_k(x) = \sum_{\alpha}  |s_\alpha(x)|_h^2$,
where $\{s_\alpha\}$ is any $L^2$-orthonormal basis for $H^0(L^k)$.      Recall that the \emph{reproducing kernel} of $L^k$ is the section $K=K_k$ of $L^k\boxtimes \bar{L}^k$ on $X\times X$ given by
\[ K(y,x) = \sum_{\alpha} s_{\alpha}(y)\boxtimes \overline{s}_{\alpha}(x),\]
 and therefore \[B_k(x) = |K(x,x)|_{h}.\]
More invariantly, for $x\in X$  let $K_x$ be the section of $L^k\otimes \bar{L}_x^k$ on $X$ given by $K_x(y):=K(y,x)$, so that $(s,K_x)_{L^2}$ is an element of the line
$L_x^k$ for any section $s$ of $L^k$.  Then the defining property of the reproducing kernel is that this should equal $s(x)\in L_x^k$:
\[ s(x) = (s,K_x)_{L^2} \quad \text{ for all } s\in H^0(L^k).\]

We base our discussion of the functions $B_k$ on the work
of \cite{berman_berndtsson_sjoestrand:08:direc_approac_to_bergm_kernel} that starts by constructing a local expansion, as we describe next. For other approaches to Bergman kernels, see for example the book
\cite{ma_marinescu:07:holom_morse_inequal_bergm_kernel}.   \smallskip

Let $U$ be an open ball in $\mathbb C^n$ centred at the origin, and $\phi\colon U\to\R$ be smooth and
strictly plurisubharmonic. Setting $\omega= \frac{i}{2\pi}\partial\bar{\partial} \phi$, define an inner product on the space of smooth functions by
\[ (u,v)^2_{\phi} = \int_{U} u\bar{v}\, e^{-\phi} \frac{\omega^n}{n!},\]
and let $H(U)_{\phi}$ be the space of holomorphic functions of finite norm.  (The reader should of course think of this situation as arising when $U$ is a chart of a manifold
over which the hermitian line bundle $L|_U$ has curvature $-2\pi i\omega$ and a holomorphic
trivialisation of norm $e^{-\phi/2}$.) Let $\chi$ be a smooth cutoff function supported on $U$  that takes the value $1$ on $\frac{1}{2}U$.

\begin{defn}
We say that a sequence $\K_k,\ k\gg0,$ of smooth functions on $U\times U$ are \emph{local reproducing kernels} mod $O(k^{-N-1})$ for $H(U)_{k\phi}$ if there is a neighbourhood $U_0\subset U$ of the origin such that for any $x\in U_0$ and any  $u\in H(U)_{k\phi}$,
\[ u(x) = (\chi u,\K_{k,x})_{k\phi} + O(k^{-N-1}e^{k\phi(x)/2}) \|u\|_{k\phi},\]
where $\K_{k,x}(y):= \K_k(y,x)$, and the $O(k^{-N-1}e^{k\phi(x)/2})$ term is uniform on $U_0$. 
\end{defn}

So whereas the reproducing kernels $K_k$ are globally and uniquely defined, the local kernels $\K_{k}$ are far from unique.  The next theorem exhibits preferred ones, which will be shown later to be local approximations to the global kernels.\medskip

We need the notion of a symmetric almost sesqui-holomorphic extension $\psi\colon U\times U\to\C$ of $\phi$. That is, $\psi(x,y)=\overline{\psi(y,x)},\ \psi(x,x)=\phi(x),$
and $\bar{\del} (\psi(x,\bar{y}))$ vanishes to all orders on $\{x=y\}$,
i.e.,
$D^{\alpha} (\bar{\del}( \psi(x,\bar{y})))|_{\{x=y\}}=0$ for all $\alpha$.
(This notion is  discussed in \cite[section 2.6]{berman_berndtsson_sjoestrand:08:direc_approac_to_bergm_kernel} and \cite[section 3.3.3.1]{rubenstein:08:geomet_quant_dynam_const_space_metric};  if $\phi$ is analytic, then one can take $\psi$ to be holomorphic in the first variable and anti-holomorphic in the second.)    By shrinking $U$ if necessary we can ensure there is a $\delta>0$ such that
\begin{equation} \label{refpsi}
\Re (\psi(x,y)) \le  \phi(x)/2 + \phi(y)/2 -\delta \|x-y\|^2
\end{equation}
for $x,y\in U$ \cite[equation 2.7]{berman_berndtsson_sjoestrand:08:direc_approac_to_bergm_kernel}, \cite[lemma 3.8]{rubenstein:08:geomet_quant_dynam_const_space_metric}.  We write $f = O(k^{-\infty})$ to mean $f = O(k^{-M})$ for any $M$.

\begin{thm}[Berman-Berndtsson-Sj\"ostrand]\label{thm:BBS}
For fixed $N,r\ge 0$ there exist smooth functions $\tilde{b}_j$ defined on $U\times U$ such that
\[\K_k(y,x)= \left(k^n + \tilde{b}_1(y,x) k^{n-1} + \dots + \tilde{b}_{N+r}(y,x)k^{n-N-r}\right)e^{k\psi(y,x)}\]
is a local reproducing kernel mod $O(k^{-N-r-1})$ for $H(U)_{k\phi}$.    Each $\tilde{b}_j$ can be written as a polynomial in the derivatives $\del_x^{\alpha}\bar{\del}^{\beta}_y\psi(x,y)$, and in particular $\tilde{b}_1(x,x) = \frac{1}{2}\Scal(\omega)$.

Moreover, if $D_{x,y}$ is a differential
operator of any order in $x$ and $y$, then  $e^{-k(\phi(x)/2+\phi(y)/2)} D_{x,y}\bar{\partial}_x
\mathcal K_k$ and  $e^{-k(\phi(x)/2+\phi(y)/2)} D_{x,y} \partial_y \mathcal K_k$
are both $O(k^{-\infty})$. 

\end{thm}
\begin{proof}
 This is \cite[proposition  2.5 and section 2.4]{berman_berndtsson_sjoestrand:08:direc_approac_to_bergm_kernel}
 when $\phi$ is analytic, and generalised to the smooth case in \emph{loc.\ cit.} proposition
 2.7.
 (The cited work is stated in the case that $U$ is a ball of radius $1$, but
 the case for general radius follows immediately on rescaling; there is also an extra factor of $\pi^{-n}$ due to differences in conventions for the curvature form.)
\end{proof}

Observe that if a finite group $G$ acts on $U$ preserving $\phi$, then we can take $\psi$
and $\chi$ to be $G$-invariant, in which case $\mathcal K_k(\zeta x,\zeta x) = \mathcal K_k(x,x)$ for all $\zeta \in G$.

\chapter{Local expansion}\label{sec:local}
We next extend the local expansion of the previous section to orbifold charts, which begins by averaging the local reproducing kernels $\mathcal K_k$.   Building on the notation of the previous section, suppose that a cyclic group $G$ of order $m$ acts linearly and faithfully on $U$  and that our plurisubharmonic function $\phi$ and cutoff function $\chi$ are invariant.  The inner product on the space of functions is replaced with
\[ (u,v)_{\phi,m} = \frac{1}{m}\int_{U} u\bar{v}\,  e^{-\phi}\frac{\omega^n}{n!},\]
and $H(U)_{\phi,m}$ is the space of holomorphic functions of finite norm.   Fix a generator $\zeta$ of $G$ and suppose we are given a character of $G$ which maps $\zeta$ to a primitive $m$-th root of unity $\lambda$.     We say a function $u$ on $U$ \emph{has weight $j$} if $u(\zeta x) = \lambda^j u(x)$ for all $x\in U$.     (The reader should  now think of $U$ as arising from an orbifold chart $U\to U/G$, and the character as describing the action of $G=\mathbb Z/m$ on $L|_U$ under some (noninvariant) trivialisation;  the condition that $\lambda$ is primitive comes from the local condition for orbi-ample line bundles,  and the sections of $L^k|_U$ now arise from functions on $U$ of weight $k$ mod $m$.)

Let $\K_k= (k^n + \tilde{b}_1 k^{n-1} + \cdots + \tilde{b}_{N+r} k^{n-N-r})e^{k\psi}$ be the preferred local reproducing kernel for $H(U)_{k\phi}$ from Theorem \ref{thm:BBS} and define
\begin{equation}
  \K_k^{\average} (y,x) := \frac{1}{m}\sum_{u,v=0}^{m-1} \lambda^{k(v-u)} \K_k(\zeta^u y,\zeta^v x).\label{eq:defofKav}
\end{equation}
One sees  that $\K_k^\average$ has weight $k$ in the first variable and weight $-k$ in the second variable (as indeed it must if it is to be a section of $L^k|_U\boxtimes \bar{L}^k|_U$). 


\begin{lem}\label{lem:invariantreproducingkernel}
  $\K_k^{\average}$ is a local reproducing kernel mod $O(k^{-N-r-1})$ for the subspace of $H(U)_{k\phi,m}$ consisting of functions that have weight $k$.  
\end{lem}
\begin{proof}
Let $u$ have weight $k$. Then by the change of variables  $y'=\zeta^u y$ and the invariance of $\phi$ and $\chi$ we have, up to terms of order $O(e^{k\phi(x)/2} k^{-N-r-1})\|u\|_{k\phi}$,
\begin{eqnarray*}
 \int_{U} \chi(y)  u(y)\overline{\lambda^{k(v-u)} \K_{k,\zeta^v x}(\zeta^u y)} e^{-k\phi(y)}\frac{\omega^n}{n!} &=& \lambda^{-kv} u(\zeta^v x) \\
&=& u(x),
\end{eqnarray*}
so
\begin{eqnarray*}
(\chi u,\K_{k,x}^{\average})_{k\phi,m} &=& \frac{1}{m} \int_{U} \chi(y)  u(y) \overline{\K_{k,x}^{\average}}(y) e^{-k\phi(y)}\frac{\omega^n}{n!} \\
&=& \frac{1}{m^2} \sum_{u,v=0}^{m-1} u(x) = u(x).
\end{eqnarray*}
\end{proof}

We show in the next section that downstairs on the orbifold, the Bergman kernel $B_k$ is locally approximated by $\K_k^{\average}(x,x)e^{-k\phi(x)}$, and thus the weighted Bergman kernel $B^{\orb}_k=\sum_i c_i B_{k+i}$  is approximated by
\[ \mathcal B_k^{\orb}(x):= \sum_{i} c_i  \K_{k+i}^{\average}(x,x) e^{-(k+i)\phi(x)},\]
where, we recall, the sum is over a fixed finite index of nonnegative integers $i$.

\begin{thm}\label{thm:localexpansion}
Suppose the $c_i$ satisfy
\begin{equation}
\frac{1}{m}\sum_{i} i^p c_i =\sumskip{i^p c_i}  \quad \text{for  } 0\le u\le m-1, 0\le p\le N+r.\label{eq:conditiononalphailocal}
\end{equation}
Then there is a local $C^0$-expansion of $\B_k^{\orb}$ of order $N+r$
\[ \B_k^{\orb}(x) = b_0(x)k^n + b_1(x) k^{n-1} + \cdots + b_{N+r}(x) k^{n-N-r} + O(k^{n-N-r-1})\]
on $U$.  Moreover, the $b_j$ depend  only on $c_i$ and the derivatives of
the metric; in particular, $b_0 = \sum_{i} c_i$ and $b_1 = \sum_{i} c_i (ni + \frac{1}{2}\Scal(\omega))$.
\end{thm}

Our proof requires a reformulation of the condition made on the $c_i$.

\begin{lem}\label{lem:lemmaonconstants}
The constants $c_i$ satisfy \eqref{eq:conditiononalphailocal} 
if and only if the function $\sum_i c_i z^i$ has a root of order $N+r+1$ at every $m$-th root of unity other than $1$.
\end{lem}
\begin{proof}
If \eqref{eq:conditiononalphailocal} holds, then for each $0\le p\le N+r$ the quantity
$ c= \sumskip{i^p c_i}$ is independent of $u$.  Hence if $\sigma^m=1$ with
$\sigma\neq 1$, then
$$
  \sum_{i} i^{p}c_i\sigma^i=
\sum_{u=0}^{m-1}  \sumskip{i^p c_i \sigma^i} = \sum_{u=0}^{m-1}  \sigma^u \sumskip{i^p c_i}  = c\sum_{u=0}^{m-1} \sigma^{u}=0,
$$
proving that $\sum_i c_i z^i$ has a root of order $N+r+1$ at $\sigma$.  For
the converse, let $\sigma$ be a primitive root of unity, so the hypothesis is that $\sum_i i^p c_i \sigma^{si}=0$ for $1\le s\le m-1$ and $0\le p\le N+r$, so for any given $u$,
\[\frac{1}{m} \sum_i i^p c_i = \frac{1}{m}\sum_{s=0}^{m-1} \sigma^{-su} \sum_i i^p c_i \sigma^{si}=\frac{1}{m}\sum_i i^p c_i\sum_{s=0}^{m-1} \sigma^{(i-u)s  }=\sumskip{i^p c_i}. \]
\end{proof}

\begin{proof}[Proof of Theorem \ref{thm:localexpansion}]
Write
\begin{eqnarray*}
  \B_k^{\orb}(x) &=& \frac{1}{m}\sum_{i} c_i \sum_{u,v=0}^{m-1}  \lambda^{(k+i)(v-u)} \K_{k+i}(\zeta^u x,\zeta^v x) 
e^{-(k+i)\phi(x)} \\
&=& S_1 + S_2,
\end{eqnarray*}
where $S_1$ consists of the terms with $u=v$ and $S_2$ consists of the terms with $u\neq v$.  We show below that  $S_2$ is $O(k^{n-N-r-1})$. Observe that 
\begin{eqnarray*}
  S_1 &=& \frac{1}{m} \sum_{i} c_i \sum_{u=0}^{m-1} \K_{k+i}(\zeta^u x, \zeta^u x) e^{-(k+i)\phi(x)} \\
&=& \sum_{i} c_i  \K_{k+i}(x, x) e^{-(k+i)\phi(x)},
\end{eqnarray*}
by the invariance of $\K_{k+i}$, and since each $\K_{k+i}e^{-(k+i)\phi}$ can be expanded in $k$, the same is true of $S_1$. In fact, by the binomial expansion $(k+i)^{n-j} = \sum_{u=0}^{N+r}\binom{n-j}{u} k^{n-j-u} i^u + O(k^{n-N-r-1})$,
\begin{eqnarray*}
  S_1&=& \sum_{i} c_i \sum_{j=0}^{N+r} (k+i)^{n-j} \tilde{b}_j(x,x)\\
&=& \sum_{j=0}^{N+r} b_j(x) k^{n-j} + O(k^{n-N-r-1}),
\end{eqnarray*}
 where 
\[ b_j (x) = \sum_{q=0}^j \left(\tilde{b}_q(x,x) \binom{n-q}{j-q} \sum_{i} c_i i^{j-q}\right).\]
In particular, $b_0$ and $b_1$ are as claimed in the statement of the theorem. \medskip   

Now to bound $S_2$,  write $S_2 = \frac{1}{m}\sum_{u\neq v} S_{u,v}$, where
\[ S_{u,v} = \sum_{i} c_i \K_{k+i}(\zeta^u x, \zeta^v x) \lambda^{(k+i)(v-u)} e^{-(k+i)\phi(x)}.\]
Fixing $u\neq v$, let $\sigma = \lambda^{v-u}$ and note that since
$\lambda$ is primitive, $\sigma\neq 1$.  Furthermore, set \[\eta = \eta(x)= e^{\psi(\zeta^u x,\zeta^v x) - \phi(x)},\]
so
\begin{eqnarray*}
 S_{u,v} &=&  \sum_{i}c_i \sum_{j=0}^{N+r} (k+i)^{n-j} \tilde{b}_j(\zeta^u x,\zeta^v x)  \sigma^{k+i} \eta^{k+i}\\
&=& \sigma^k \sum_{j=0}^{N+r} k^{n-j} \sum_{q=0}^j \tilde{b}_q(\zeta^u x, \zeta^v x) \binom{n-q}{j-q} \sum_{i} c_i i^{j-q} \sigma^i \eta^{k+i} \\
&&+ O(k^{n-N-r-1}).
\end{eqnarray*}
Therefore, to prove the theorem it is sufficient to show that for $0\le l\le N+r$, 
\begin{equation}\label{eq:suvbound}
  u_{l,k}:=\sum_{i} c_i i^{l} \sigma^i \eta^{k+i} = O(k^{l-N-r-1}).
\end{equation}
To this end, write
\begin{eqnarray*}
u_{l,k} &=&  \left[\frac{\sum_i c_i i^{l} \sigma^i \eta^i}{(\eta-1)^{N+r-l+1}}\right] (\eta-1)^{N+r-l+1}\eta^k.
\end{eqnarray*}
 From Lemma \ref{lem:lemmaonconstants},  the function $\eta \mapsto \sum_i
 c_i i^{l} \sigma^i \eta^i$ has a root of order $N+r-l+1$ at $\eta=1$, and so the term in square brackets is bounded. So it is sufficient to prove the following:\medskip

\noindent \emph{Claim: } Let $s\ge 1$ and $u\neq v$.  Then 
\[(\eta-1)^s \eta^k = O(k^{-s}).\]  
\noindent \emph{Proof: }  Since $G$ acts linearly on $U$, we can write
$U=U_1\times U_2$, where the action is faithful on $U_1$ and trivial on
$U_2$.  For $x\in U$, write $x=(x_1,x_2)$ under this decomposition.  Then there exist positive constants $c,c'$ such that $c\|x_1\|^2 \le \| \zeta^ux-\zeta^vx\|^2 \le c'\|x_1\|^2$ for all $x\in U$. Hence from \eqref{refpsi}, there is a $\delta'>0$ such that
\[ \Re(\psi(\zeta^u x,\zeta^v x)-\phi(x)) \le -\delta' \|x_1\|^2\]
for all $x\in U$.  If $x_1=0$, then $\eta=1$ and we are done.  Assuming
$x_1\neq 0$,  recall that if $z$ is a complex number with $\Re z<0$, then $|e^{kz}| \le \frac{s!}{(-k\Re z)^s}$ for all $k$.    Thus there is a constant $C$ such that
\[ \left \vert (\eta(x)-1)^s \eta(x)^k \right\vert\le \frac{C}{k^s}\left\vert\frac{\eta(x)-1}{\|x_1\|^2} \right\vert^s.\]
But $\psi(\zeta^u x,\zeta^v x)-\phi(x)=O(\|x_1\|^2)$, so $\frac{\eta(x)-1}{\|x_1\|^2}$ is bounded and the claim follows.
\end{proof}

\begin{rmk}\label{rmk:necessityofconstants}
  Conversely, suppose that $\mathcal B_k^{\orb}$ admits an asymptotic expansion in $C^0$ of order $N$ at the   point $x=0$ in $U$ which is fixed by the group action. We will show that
the $c_i$ must satisfy the conditions \eqref{eq:conditiononalphailocal} for $r=0$.

Since
  $\lambda$ is primitive and sections of orbi-ample line bundles
  vanish at orbifold points,
$$
\K_{k+i}^{\average}(0,0) = \left\{
  \begin{array}{ll}
    m\K_{k+i}(0,0) & \text{ if } k+i\equiv 0 \text{ mod } m, \\
    0 & \text{ otherwise.}
  \end{array}\right.
$$
Therefore
\begin{eqnarray*}
  \B_k^{\orb}(0) &=& m\sumskipp{i}{-k}{c_i \K^{\average}_{k+i}(0,0)}=m \sumskipp{i}{-k}{c_i \sum_{p=0}^N \tilde{b}_p(0,0) (k+i)^{n-p}} \\
  &=& \sum_{p=0}^N b_p k^{n-p} + O(k^{n-N-1}),
\end{eqnarray*}
where
\begin{equation} \label{bpk} b_p = \sum_{q=0}^p
  \left(\tilde{b}_q(0,0)\binom{n-q}{p-q} \sumskipp{i}{-k} c_i
    i^{p-q}\right).
\end{equation}
Each $b_p$ is periodic in $k$; the existence of an asymptotic
expansion of $\mathcal B_k^{\orb}(0)$ implies that in fact each $b_p$ is independent of $k$.

We now use induction on $p$ to show the conditions
\eqref{eq:conditiononalphailocal}, i.e.\ that $\sumskipp{i}{-k} c_i
i^p$ is also independent of $k$. For $p=0$, the sum \eqref{bpk} reduces
to $b_0=\sum_{i\equiv-k}c_i$ since $\tilde{b}_0(0,0)=1$. Therefore,
$b_0$'s independence of $k$ implies the $p=0$ case of
\eqref{eq:conditiononalphailocal}. For general $p$, the sum \eqref{bpk}
is $\binom np\sumskipp{i}{-k} c_i i^p$ plus terms shown inductively to
be independent of $k$. Therefore, we recover the conditions
\eqref{eq:conditiononalphailocal} for $r=0$.
\end{rmk}


\chapter{Global Expansion}\label{sec:global}

We will require the following standard estimate, which we prove for completeness. At
first we do not require $X$ to be compact.
 
\begin{lem}
Let $X$ be a K\"ahler manifold of dimension $n$.  There exists a constant
$C$ such that for any $p\in X$ we have $|f(p)|\le C(a\|\bar\partial f\|_{C^0} + a^{-n}\|f\|_{L^2})$ for all smooth functions $f$ and all sufficiently small $a>0$.  Moreover, if $X$ is compact then we can choose $a>0$ uniformly over all $p\in X$.
\end{lem}
 
\begin{proof}
We first prove the one dimensional case. Let $\epsilon =\sup|\bar\partial f|$, pick a local coordinate $z=r e^{i\theta}$ about the point $p=0$, and let $D_a$ denote
the ball $|z|=r\le a$ in this coordinate.  By the fundamental theorem of calculus,
\begin{equation} \label{UBC}
\frac1{2\pi a}\int_{\partial D_a}\frac{\partial
f}{\partial\theta}d\theta=0.
\end{equation}
We know that
$$
\bar\partial f=\frac{1}{2}\left(\frac{\partial f}{\partial r}+\frac
ir\frac{\partial
f}{\partial\theta} \right)(dr-ird\theta)
$$
has a pointwise bound on its norm of $O(\epsilon)$. Here we measure
norms of 1-forms in the standard metric, which by the compactness of $X$ is boundedly
close to the K\"ahler metric in $D_a$.
 
Then $\frac 1r\frac{\partial f}{\partial\theta}$ is within $\epsilon$ of
$i\frac{\partial f}{\partial r}$, and (\ref{UBC}) gives
$$
\frac1{2\pi}\left|\int_{\partial D_a}\frac{\partial f}{\partial
r}d\theta\right|=
\left|\frac1{2\pi}\frac d{da}\int_{\partial
D_a}fd\theta\right|\le\epsilon.
$$
Integrating with respect to $a$ gives the estimate
$$
|f(0)|\le\frac1{2\pi}\left|\int_{\partial D_a}fd\theta\right|+a\epsilon.
$$
Multiplying by $a$ and integrating again yields
$$
\frac{a^2}2|f(0)|\le\frac1{2\pi}\int_{D_a}|f|rdrd\theta+\frac{a^3}3\epsilon.
$$
By the Cauchy-Schwarz inequality, this gives the desired bound on
$|f(0)|$
in terms of the $L^2$-norm of $f$ and
the pointwise supremum of $|\bar\partial f|$.\medskip

Now we pass to the general case. Pick local coordinates $z_i$ in which
$p$ is the origin, and apply the above argument over one dimensional discs
in the $z_n$ direction to bound
\begin{equation} \label{line1}
|f(z_1,\ldots,z_{n-1},0)|\le\frac1{\pi a^2}\int_{|z_n|\le
a}|f(z_1,\ldots,z_{n_1},z_n)|+\frac{2a}3\epsilon.
\end{equation}
Similarly,
$$
|f(z_1,\ldots,z_{n-2},0,0)|\le\frac1{\pi a^2}\int_{|z_{n-1}|\le
a}|f(z_1,\ldots,z_{n_2},z_{n-1},0)|+\frac{2a}3\epsilon,
$$
which by (\ref{line1}) can be bounded by 
$$
|f(z_1,\ldots,z_{n-2},0,0)|\le\frac1{\pi^2a^4}\int_{|z_n|,|z_{n-1}|
\le a}|f(z_1,\ldots,z_n)|+\frac{2a}3\epsilon+\frac{2a}3\epsilon.
$$
Inductively, we find that
$$
|f(0)|\le\frac1{\pi^na^{2n}}\int_{|z_i|\le
a}|f(z_1,\ldots,z_n)|+\frac{2na}3\epsilon.
$$
Again an application of Cauchy-Schwartz gives the result. 
\end{proof}

\begin{cor}\label{cor:schwarzlemma}
Let $X$ be a compact K\"ahler orbifold and $L$ be an orbifold line bundle with hermitian
metric.  There exists a constant $C$ such that for any section $s\in\Gamma(L^k)$ and
$x\in X$ we have  $|s(x)|\le C(k^{-\frac12}\|\bar\partial s\|_{C^0} + k^{\frac n2}
\|s\|_{L^2})$.
\end{cor}
 
\begin{proof}
This follows from the previous result, since we can work locally in an orbifold chart where
upstairs we have a smooth K\"ahler metric. We trivialise $L$ (and so each $L^k$) upstairs with a holomorphic section of norm $1$ at $x$ which is possibly not invariant under $G$. In a ball of radius $a=k^{-1/2}$ about $x$, the hermitian metric is then boundedly close to taking absolute values. Picking the coordinates in the above proof to be
invariant under the finite group gives, for $k\gg0$, a bound in terms of an integral upstairs over a $G$-invariant ball.
The actual integral on the orbifold differs from this by dividing by the
order of the group; since this is finite we get the same order bound.
\end{proof}

We now apply the results from the previous section to orbifold charts.  As usual, $X$ is an orbifold with cyclic quotient singularities, and $L$ is an orbi-ample line bundle with hermitian metric with curvature $-2\pi i\omega$.  Suppose that $U\to U/G$ is a small orbifold chart in $X$, where $U$ is a small ball centred at the origin in $\mathbb C^n$ and $G$ is cyclic of order $m$. By the orbi-ampleness condition, $L$ has a trivialisation
with weight $-1$ under the action of $G$ via an identification between $G$ and the $m$th roots of unity. Then sections of $L^k|_U$ are given locally by holomorphic functions $f$ of weight $k$. The pointwise norm of such a section is $|f(z)|e^{-k\phi(z)/2}$, where
the plurisubharmonic function $\phi$ is the norm squared of the trivialising section.

As previously mentioned, the local reproducing kernel
$\K_k^{\average}(y,x)$ on $U\times U$ defined in \eqref{eq:defofKav} has weight $k$ in the first variable, and weight $-k$ in the second and $\chi$ is an invariant cutoff function supported on $U$ and identically 1 on $\frac{1}{2}U$.  Thus multiplying by the local trivialisation, we think of $\chi \K_{k,x}^{\average}$ as a smooth section of $L^k\otimes \bar{L}^k_x$.    The local reproducing property for $\K_k^{\average}$ proved in Lemma \ref{lem:invariantreproducingkernel} says there is a neighbourhood $U_0\subset U$ of the origin such that if $x\in U_0$ and $t\in H^0(L^k)\otimes \bar{L}^k_x$, then
\[ t(y) = (t,\chi \K^{\average}_{k,y})_{L^2} + O(k^{-N-r-1})\|t\|_{L^2} \quad \text{for } y\in U_0,\]
where the error term is measured using the hermitian metric on
$L_y^k\otimes \bar{L}_x^k$.   In what follows, $U$ will be fixed, but $U_0$
will be allowed to shrink as required.    If $s\in H^0(L\boxtimes \bar{L})$, then applying this to $t=s_x=s(\ \cdot\ ,x)$ gives
\begin{equation}
 s(y,x) = (s_x, \chi \K^{\average}_{k,y})_{L^2} + O(k^{-N-r-1})\|s_x\|_{L^2} \label{eq:reproducingpropertyforglobal}.
\end{equation}

Now recall that $K_k=\sum_\alpha s_\alpha\boxtimes\overline{s}_\alpha$ denotes the global reproducing kernel discussed at the start of Section \ref{sec:bergman}. The next proposition shows how this is approximated by the local reproducing kernels $\K^{\average}_k$.

\begin{prop}\label{prop:global}
There is a neighbourhood $U_0\subset \frac14U$ of the origin such that
\[ K_k|^{}_{U_0\times U_0}= \K_k^{\average} + O(k^{n-N-r-1}).\]
\end{prop}

\begin{proof}
This is identical to the manifold case \cite[theorem 3.1]{berman_berndtsson_sjoestrand:08:direc_approac_to_bergm_kernel}, and for convenience we sketch the details.  The Bergman kernel has the extremal characterisation $B_k(x) = \sup_{\|s\|=1} |s(x)|^2$ where the supremum is over all sections of $L^k$ of unit $L^2$-norm.  Using this, it is shown in \cite[theorem 1.1]{berman:04:bergm_kernel_local_holom_morse_inequal} that there is a constant $C$ such that  $B_k(x)\le Ck^n$ uniformly on $X$ (the cited  work is for  noncompact manifolds; what is important is that it is local, so the inequality only improves if we apply it upstairs on an orbifold chart and restrict to invariant sections).

By the Cauchy-Schwarz inequality applied to $K_k= \sum_{\alpha} s_{\alpha}\boxtimes \bar{s}_{\alpha}$, we have $|K_k(y,x)|_h \le \sqrt{B_k(x) B_k(y)} = O(k^n)$, so in particular $\|K_{k,x}\|_{L^2} = O(k^n)$ uniformly in $x$.   Putting $s:=K_k$ (which is holomorphic) into the local reproducing property, \eqref{eq:reproducingpropertyforglobal} gives
\begin{eqnarray*}
K_k(y,x) &=& (K_{k,x}, \chi\K_{k,y}^{\average})_{L^2} + O(k^{n-N-r-1}) \nonumber 
\end{eqnarray*}
on $U_0\times U_0$. Then for $x,y\in U_0$ we can write
\begin{eqnarray}
  K_k(y,x)&=&  \K^{\average}_{k}(y,x) - \overline{w(y,x)} + O(k^{n-N-r-1}),
  \label{eq:reproducelocalglobal}
\end{eqnarray}
where
\begin{eqnarray*}
 w(y,x) &:=& \overline{\K^{\average}_{k}(y,x)} - (\chi \K^{\average}_{k,y},K_{k,x})_{L^2}\\
&=& \chi(x) \K^{\average}_{k,y}(x) - (\chi \K^{\average}_{k,y},K_{k,x})_{L^2},
\end{eqnarray*}
as $\K_k^{\average}$ is hermitian. With the aim of bounding $w$, define $u_y(\,\cdot\,): = w(y,\,\cdot\,)$ so
 \[u_y(\,\cdot\,) = \chi \K^{\average}_{k,y} -  (\chi
 \K^{\average}_{k,y},K_{k,\,\cdot\,})_{L^2}.\]  Notice that the inner product in this expression is precisely the projection of $ \chi \mathcal K^{\average}_{k,y}$ onto the space of holomorphic sections of $L^k$.  So, said another way,  $u_y$ is the $L^2$-minimal solution of the equation
\begin{equation}
\bar{\partial} u_y = \bar{\partial}( \chi \K^{\average}_{k,y}). \label{eq:dbarequation}
\end{equation}
Now 
\begin{equation*}
\bar{\partial} (\chi \K^{\average}_{k,y})(t) =  \K^{\average}_{k,y}(t) \bar{\partial} \chi(t) + \chi(t) \bar{\partial} \K^{\average}_{k,y}(t),\label{eq:delbarxiK}
\end{equation*}
which we claim is $O(k^{-\infty})$ in $C^0$.  Since we are assuming that
$U_0\subset \frac{1}{4}U$,  $\chi$ is identically $1$ on $2U_0$ and supported on $U$ so  $\bar{\partial}\chi(t)$ vanishes for $t\in 2U_0$ and for $t\notin U$.  On the other hand, by \eqref{refpsi} there is a $\delta_1>0$ such that
\[\Re[2\psi(t,y) - \phi(t)-\phi(y)]\le -\delta_1 \|t-y\|^2\]
on $U\times U$. If $t\in U\backslash 2U_0$, so that $t$ is a bounded
distance away from $y$, then by Theorem \ref{thm:BBS},
\begin{eqnarray*}
 |\mathcal K_{k,y}(t)|_h^2 &=& O(k^{2n}) e^{2k\psi(t,y)} e^{-\phi(t)} e^{-\phi(y)} \\
&=& O(k^{2n}) e^{k[2\psi(t,y) - \phi(t)-\phi(y)]} \\
&\le& O(k^{2n}) e^{-\delta_2 k}\ =\ O(k^{-\infty}).
\end{eqnarray*}

Thus $ \mathcal K_{k,y}\bar{\partial}\chi= O(k^{-\infty})$ in $C^0$ on all
of $X$, and furthermore the $O(k^{-\infty})$ term is independent of $y\in U_0$.  Applying this to $\xi^v y$ as $v$ ranges over a period, we get $ \mathcal K^{\average}_{k,y}\bar{\partial}\chi = O(k^{-\infty})$ as well.  Now from the second statement in Theorem \ref{thm:BBS}, $\chi \bar{\partial}  \mathcal K^{\average}_{k,y} = O(k^{-\infty})$, hence $\bar{\partial} (\chi \K^{\average}_{k,y}) = O(k^{-\infty})$ in $C^0$, and therefore in $L^2$ as well.  \medskip

Thus applying the H\"ormander estimate to sections of $L\otimes \bar{L}_y$, we conclude
that $\| u_y \|_{L^2}  = O(k^{-\infty})$.
Using this and that the bound on $\bar{\partial} u_y$ is uniform in $x$ and $y$, Corollary \ref{cor:schwarzlemma} gives the pointwise estimate
\[ |w(y,x)|=|u_y(x)| = O(k^{-\infty})\quad \text{ on }U_0\times U_0.\]
Therefore \eqref{eq:reproducelocalglobal} becomes $K_k = \K^{\average}_k + O(k^{n-N-r-1})$ on $U_0\times U_0$, as required.
\end{proof}

So $\mathcal K_k^{\average}$ approximates $K_k$ to order $O(k^{n-N-r-1})$ in the $C^0$ norm.  We next use this to get a $C^r$ expansion at the expense of a factor of $O(k^r)$.

\begin{lem}\label{lem:cauchyestimate}
Suppose $f_k(x,y)$ is a sequence of functions on $U_0\times U_0$ such that $(\bar{\partial}_x f_k) e^{-k(\phi(x)/2 + \phi(y)/2)}$ and $(\partial_y f_k) e^{-k(\phi(x)/2 + \phi(y)/2)}$ are both of order $O(k^{-\infty})$.   Suppose also that $f_k(x,y) e^{-k(\phi(x)/2 + \phi(y)/2)} = O(k^q)$ uniformly on $U_0\times U_0$.
Then for any differential operator in $x$ and $y$ of total order $p$,
$$(D f_k) e^{-k(\phi(x)/2 + \phi(y)/2)} = O(k^{q+p}) \text{ on } \frac{1}{2}U_0\times \frac{1}{2}U_0.$$
\end{lem}
\begin{proof}
Fix $(x,y)\in \frac{1}{2}U_0\times \frac{1}{2}U_0$.  Consider first the
case that $f(x,y)=f_k(x,y)$ is holomorphic in the first variable and anti-holomorphic in the second. Suppose $x=(x_1,\ldots x_n)$, $y=(y_1,\ldots,y_n)$ and let $F(t) = f(x_1,\ldots, t,\ldots,x_n, y_1,\ldots, y_n)$ and similarly $\Phi(t) =
\phi(x_1,\ldots,t,\ldots,x_n)$, where the $t$ lies in the $i$-th coordinate.  If $B$ denotes the ball of radius $k^{-1}$ around $x_i$, then for $k$ sufficiently large $(x_1,\ldots,t,\ldots,x_n)\in U$ for all $t\in B$.  So applying the Cauchy formula to $B$ gives
\begin{eqnarray*}  
  \left.\frac{\del f}{\del x_j}\right|_{(x,y)}\!\!\!\!\!\!\! \!\!\!e^{-k\left(\frac{\phi(x)}{2} +\frac{\phi(y)}{2}\right)} &=\!\! &
\int_{\partial B} \!\!\!\frac{F(t) e^{-k(\Phi(t)/2+\phi(y)/2)}}{2\pi i (t-x_j)^2}  e^{-k(\phi(x)/2 - \Phi(t)/2)}dt \\
&=& \int_{\partial B} \frac{O(k^q)}{(t-x_j)^2} dt\\
&=& O(k^{q+1}),
\end{eqnarray*}
where the second inequality uses $\phi(x) - \Phi(t) = O(|x_j-t|) = O(k^{-1})$ on $\partial B$, so the exponential term is $O(1)$.    The other derivatives are treated similarly.

More generally, if we assume only $\bar{\partial}_x f = O(k^{-\infty})$, then the calculation above holds up to a term $e^{-k(\phi(x)/2+\phi(y)/2)}\int_{B} \del F/\partial \bar t\, dt d\bar{t}$ which is of order $O(k^{-\infty})$, so the conclusion still holds.
\end{proof}

Recall that the Bergman kernel is $B_k(x) = |K_k(x, x)|_h$.

\begin{cor}\label{cor:approxofbergman}
For $k$ sufficiently large,
 $$B_k(x) = \mathcal K^{\average}_{k}(x,x) e^{-k\phi(x)} + O(k^{n-N-1})$$ in
 $C^r$ on $\frac{1}{2}U_0$.  Moreover, if $D$ is a differential operator of
 order $p$, then
\begin{equation}
D B_k = O(k^{n+p}) \text{ in } C^0 
\end{equation}
uniformly on $X$.
\end{cor}
\begin{proof}
Let $g_k$ be the local expression of $K_k$ in our trivialisation of $L|_U$, so $B_k(x) = |K_k(x,x)|_h = g_k(x,x) e^{-k\phi(x)}$.   Now Proposition \ref{prop:global} restricted to the diagonal implies $(g_k-\mathcal K_k^{\average})e^{-k\phi(x)} = O(k^{n-N-r-1})$ in $C^0$ on $U_0$.  So we can apply Lemma \ref{lem:cauchyestimate} to $f_k:=g_k-\mathcal K_k^{\average}$ to deduce that if $D$ is a differential operator of order $q$ in $x$ and $y$ then 
$$D(g_k - \mathcal K_k^{\average}) e^{-k(\phi(x)/2+\phi(x)/2)} = O(k^{n-N-r-1+q})$$
on $\frac{1}{2}U_0$.   Therefore, taking the first $r$ derivatives of
 $$
 B_k(x) - \mathcal K_k^{\average}(x,x)e^{-k\phi(x)} =
 \big(g_k(x,x)-\mathcal K_k^{\average}(x,x)\big)e^{-k\phi(x)}
 $$
 with respect to $x$ using the Leibnitz rule proves the first statement.\medskip

For the second statement,  Proposition \ref{prop:global} implies $g_ke^{-k(\phi(x)/2 + \phi(y)/2)}$ is $O(k^n)$ on $U_0\times U_0$. Thus from  Lemma
\ref{lem:cauchyestimate} we get that if $D$ has order at most $q$ then $(Dg_k) e^{-k(\phi(x)/2 + \phi(x)/2)} = O(k^{n+q})$ on $\frac{1}{2}U_0$.  Hence taking the derivatives of $B_k(x) = g_k(x,x) e^{-k\phi(x)}$ with the Leibnitz rule gives $DB_k = O(k^{n+p})$ on $\frac{1}{2}U_0$ for any differential operator of order $p$.  So by compactness the same holds uniformly on all of $X$.
\end{proof}

The proof of our main theorem is now immediate.   Each $x$ is contained in some open set $U_0$ for which Theorem \ref{thm:localexpansion} and Corollary
\ref{cor:approxofbergman}  hold, and by compactness there is a finite cover by such sets.   Thus
  \begin{eqnarray*}
    B_k^{\orb}(x) &=& \sum_{i} c_i B_{k+i}(x) 
= \sum_{i} c_i \K_{k+i}^{\average}(x,x) e^{-(k+i)\phi(x)} + O(k^{n-N-1}) \\
&=& \sum_{j=0}^N b_j(x) k^{n-j} + O(k^{n-N-1})
  \end{eqnarray*}
in $C^r$, as required.  \medskip

In \cite{ross_thomas:stabil_orbif} we use a variant of our main theorem, which for convenience we record here.

\begin{cor}\label{cor:homogeneousweighting}
Let $\gamma$ be a homogeneous polynomial in two variables of degree $d\le N$, and suppose that as usual the $c_i$ are chosen to satisfy \eqref{eq:conditiononalpha}.   Then there is a $C^{r}$-expansion
\[  \sum_{i} c_i\gamma(k,i) B_{k+i} = b_0 k^{n+d} + b_1 k^{n+d-1} + \cdots + O(k^{n-N+d-1}).\]
Moreover, if $\gamma(i,k) = Ak^d + B k^{d-1}i + \cdots$ then
$$
  b_0 =A \sum_{i} c_i \quad\text{ and } \quad   b_1 = \sum_{i} c_i (A[ni+\Scal(\omega)/2] + iB).
$$

\end{cor}
\begin{proof}
By linearity we may suppose $\gamma(i,k)= i^ak^{d-a}$ for some $a\le d$.    Since $c_i$ satisfy \eqref{eq:conditiononalpha} for $p=0,\ldots,N+r$, we get that $c_i':=i^ac_i$ satisfy this condition for $p=0,\ldots,N+r-a$.  Applying our main theorem with the constants $c_i'$ proves the corollary.
\end{proof}

\begin{rmk}
For example,
\[ \sum_{i} (k+i) c_i B_{k+i} = b_0 k^{n+1} + b_1 k^n + \cdots\]
where $b_0 = \sum_{i} c_i$ and $b_1 = \sum_{i} c_i[(n+1)i+ \Scal(\omega)/2]$.
\end{rmk}

\clearpage
\bibliography{orbibergman}
\begin{flushright}

{\scshape DPMMS, University of Cambridge, Cambridge, CB3 0WB, UK}
{\tt j.ross@dpmms.cam.ac.uk}\medskip

{\scshape Dept. of Mathematics, Imperial College, London, SW7 2AZ, UK}
{\tt richard.thomas@imperial.ac.uk}

\end{flushright}

\end{document}